\documentclass[reqno,12pt]{amsart}

\usepackage{color}
\usepackage{enumerate}

\numberwithin{equation}{section}

\theoremstyle{plain}
\newtheorem{thm}{Theorem}[section]
\newtheorem{prop}[thm]{Proposition}

\theoremstyle{remark}

\theoremstyle{definition}

\newtheorem{defi}[thm]{Definition}

\newtheorem{ques}[thm]{Question}

\newcommand\eps{\epsilon}

\newcommand\Ic{{\mathcal{I}}}
\newcommand\Bc{{\mathcal{B}}}

\newcommand\Reals{{\mathbb R}}

\newcommand\Nats{{\mathbb N}}

\begin{document}

\title[H\"older's inequality]{H\"older's inequality for roots of symmetric operator spaces}

\author[Dykema]{Ken Dykema$^{1}$}
\address{K.D., Department of Mathematics, Texas A\&M University,
College Station, TX 77843-3368, USA}
\email{kdykema@math.tamu.edu}

\author[Skripka]{Anna Skripka$^2$}
\address{A.S., Department of Mathematics and Statistics,
	University of New Mexico, 400 Yale Blvd NE, MSC01 1115, Albuquerque, NM 87131, USA}
\email{skripka@math.unm.edu}

\thanks{\footnotesize ${}^1$Research supported in part by NSF grant DMS--1202660.
${}^2$Research supported in part by NSF grant DMS--1249186.}


\date{May 19, 2015}

\begin{abstract}
We prove a version of H\"older's
inequality with a constant
for $p$-th roots of symmetric operator spaces of operators affiliated to a semifinite von Neumann algebra factor,
and with constant equal to $1$ for strongly symmetric operator spaces.
\end{abstract}

\maketitle

\section{Introduction}
Let $\Bc$ be an infinite dimensional, $\sigma$--finite von Neumann algebra factor
equipped with a normal, faithful, semifinite (or finite) trace $\tau$ and acting on a Hilbert space.
A closed, densely defined,
unbounded operator $T$ on the Hilbert space is affiliated with $\Bc$ if the partial isometry from its polar decomposition and all spectral projections of its
absolute value lie in $\Bc$, and it is said to be $\tau$-measurable if for every $\eps>0$, there
is a projection $p\in\Bc$ such that $\tau(p)<\eps$ and $T(1-p)$ is bounded.
$\Bc$ together with the set of
$\tau$-measurable operators $T$ as described above is a $*$-algebra under natural operations (performing addition and multiplication
with appropriate domains, and taking closures).
(See, for example,~\cite{LSZ} for details.)
We let $S(\Bc,\tau)$ denote this $*$-algebra.
Note that, when $\Bc$ is a type I von Neumann algebra, then $S(\Bc,\tau)=\Bc$.

We will consider subspaces $\Ic\subseteq S(\Bc,\tau)$ that are $\Bc$-bimodules and such that there is a complete symmetric norm $\|\cdot\|_\Ic$ on $\Ic$ (see Definition \ref{symmetric}).
Such a pair $(\Ic,\|\cdot\|_\Ic)$ is called a {\it symmetric operator space}~\cite{LSZ}.
Note that if $\Ic\subseteq\Bc$, then $\Ic$ is an ideal of $\Bc$.
Among the symmetric operator spaces are the
fully symmetric
Schatten-von Neumann ideals $S^p$, noncommutative $L^p$ spaces, $1\leq p<\infty$, and the Marcinkiewicz (or Lorentz) operator spaces
\begin{equation}\label{eq:Mcz}
\Ic_\psi:=\left\{A\in S(\Bc,\tau):\, \|A\|_{\Ic_{\psi}}:=\sup_{t>0}\frac{1}{\psi(t)}\int_0^t\mu_s(A)\,ds<\infty\right\},
\end{equation}
where $\psi$ is a concave function satisfying
\[\lim_{t\rightarrow 0^+}\psi(t)=0,\quad \lim_{t\rightarrow \infty}\psi(t)=\infty\]
and $\mu(A)$ is the generalized singular value function
of $A$ (see \eqref{gsv}).

For $1<p<\infty$, define the set
\begin{equation}
\label{rootideal}
\Ic^{1/p}=\{A\in S(\Bc,\tau):\, |A|^p\in\Ic\}
\end{equation}
and the function
\begin{equation}
\label{rootnorm}
\|A\|_{\Ic^{1/p}}:=\||A|^p\|_{\Ic}^{1/p},\quad\text{for }\,A\in\Ic^{1/p}.
\end{equation}
The space $\Ic^{1/p}$ endowed with the function $\|\cdot\|_{\Ic^{1/p}}$ is also a symmetric operator space
(see Theorem \ref{correspondence}).
When $\Ic$ is the noncommutative $L^1$ space $\{A\in S(\Bc,\tau):\, \tau(|A|)<\infty\}$, then $\Ic^{1/p}$
is the noncommutative $L^p$ space and the noncommutative H\"older inequality
\begin{align}
\label{Hi}
\|AB\|_{\Ic}\leq\|A\|_{\Ic^{1/p}}\|B\|_{\Ic^{1/q}},\quad(A\in\Ic^{1/p},\; B\in\Ic^{1/q},\;
1<p,q<\infty,\;\frac1p+\frac1q=1),
\end{align}
holds~\cite[Theorem 4.2]{FK}.
By \cite[Proposition 2.5]{DS},
the H\"{o}lder inequality~\eqref{Hi}
holds when $\Ic$ is a Marcinkiewicz operator space 
of a $\sigma$-finite semifinite von Neumann algebra factor $\Bc$, with norm
$\|A\|_{\Ic_\psi}$ as indicated in~\eqref{eq:Mcz}
(and also --- see the erratum to~\cite{DS} --- for $\Ic_\psi\cap\Bc$
with the symmetric norm
$\max\{\|\nolinebreak\cdot\nolinebreak\|,\|\nolinebreak\cdot\nolinebreak\|_{\Ic_\psi}\}$.)

In this note (see Theorem \ref{Holderideal}), we show that a weaker H\"{o}lder-type inequality
\begin{align}\label{H4}
\|AB\|_{\Ic}\leq 4\,\|A\|_{\Ic^{1/p}}\|B\|_{\Ic^{1/q}}
\end{align}
holds in an arbitrary symmetric operator space $(\Ic,\|\cdot\|_\Ic)$.
We also show that the H\"older inequality~\eqref{Hi} 
holds in an arbitrary strongly symmetric operator space (see Definition \ref{ssdef}).
The class of strongly symmetric operator spaces includes the fully symmetric operator space, and the first example of a symmetric operator space that is not fully symmetric is due to \cite{Russu}. An example of a symmetric (Banach) function space that is not strongly symmetric can be found in \cite{SSS}, and examples of  symmetric operator spaces that are not strongly symmetric can be constructed based on Theorem \ref{correspondence}.

Observe that if we take the equivalent norms $\|\cdot\|_{1/p}=2\,\|\cdot\|_{\Ic^{1/p}}$ and $\|\cdot\|_{1/q}=2\,\|\cdot\|_{\Ic^{1/q}}$
on $\Ic^{1/p}$ and $\Ic^{1/q}$, respectively,
then from~\eqref{H4} we get the inequality
\begin{align}\label{H4'}
\|AB\|_{\Ic}\leq \|A\|_{1/p}\|B\|_{1/q}.
\end{align}

Our interest in these versions of H\"older's inequality for symmetric operator spaces is motivated in part
by the second order trace formulas in \cite{DS,S},
which are proved in the case of a symmetric operator ideal $\Ic$
for perturbations in the square root ideal $\Ic^{1/2}$ that possesses a norm $\|\cdot\|_{1/2}$ satisfying
\eqref{H4'} with $p=q=2$.

\section{A H\"{o}lder-type inequality in symmetric operator spaces}

The generalized singular value function $\mu(A)$ of $A\in S(\Bc,\tau)$ is defined by
\begin{align}
\label{gsv}
\mu_t(A):=\inf\{s\geq 0:\, \tau(\chi_{(s,\infty)}(|A|))\leq t\},\quad t>0
\end{align}
(see \cite{FK}). 
It has properties analogous to those of the decreasing rearrangement
and to usual singular values of operators (see \cite[Lemma 2.5]{FK}).

\begin{defi}
\label{symmetric}
A norm on a $\Bc$-bimodule $\Ic$ of $S(\Bc,\tau)$ is called {\em symmetric} if
$A\in S(\Bc,\tau)$,
$B\in\Ic$, $\mu(A)\leq\mu(B)$ implies $A\in\Ic$ and $\|A\|_\Ic\leq\|B\|_\Ic$.
\end{defi}
We note that the properties of an operator ideal that were important in \cite{DS,S} naturally hold for an ideal in a symmetric operator space. If a $\Bc$-bimodule $\Ic\subseteq S(\Bc,\tau)$ is symmetric, then $A\in\Bc$, $B\in\Ic$, and $0\le A\le B$ imply $\|A\|_\Ic\le\|B\|_\Ic$. It is proved in \cite[Theorem 2.5.2]{LSZ} that for a Banach bimodule $(\Ic,\|\cdot\|_{\Ic})$ of $S(\Bc,\tau)$, the symmetry of the norm $\|\cdot\|_{\Ic}$ is equivalent to the property
\[\|ABC\|_\Ic\leq\|A\|\,\|B\|_\Ic\,\|C\|,\quad (A,C\in\Bc, B\in\Ic).\]
It was also assumed in \cite{DS,S} that $\Ic$ lies in $\Bc$ and the norm
has the property $\|\cdot\|\le K\|\cdot\|_\Ic$ for some constant $K>0$, where $\|\cdot\|$ is the operator norm.
If $\mathcal{J}$ is a symmetric operator space in $S(\Bc,\tau)$, then $\mathcal{J}\cap\Bc$ endowed with the norm $\max\{\|\cdot\|,\|\cdot\|_{\mathcal{J}}\}$ is a symmetric operator ideal satisfying
these properties.

There is a correspondence between symmetric operator spaces and Banach function and sequence spaces that is provided by the generalized singular values of operators.
In the case of a type I$_\infty$ factor, this is the well known Calkin correspondence.
We will recall some salient aspects here; see, for example, \cite{LSZ} for a thorough exposition.
Let
\[
J=\begin{cases}
\Nats,&\Bc\text{ a type I$_\infty$ factor} \\
[0,1],&\Bc\text{ a type II$_1$ factor} \\
[0,\infty),&\Bc\text{ a type II$_\infty$ factor,}
\end{cases}
\]
equipped with counting measure if $J=\Nats$ and Lebesgue measure otherwise;
let $D(J)$ denote the vector space of all measurable, real-valued functions $f$ on $J$ such that for every $\eps>0$, the measure of $\{t\in J\mid |f(t)|>\eps\}$
is finite (and where elements of $D(J)$ are regarded as being the same if they differ only on a set of measure zero). Let $x^*$ denote the decreasing rearrangement of $|x|$, when $x\in D(J)$.
Properties of the decreasing rearrangement can be found in, for example, \cite[Proposition 1.7]{BS}.
\begin{defi}
\label{symmetric.function.sp}
A {\em symmetric function space} (also called, when $J=\Nats$, a {\em symmetric sequence space})  is a subspace $E$ of $D(J)$
equipped with a complete norm $\|\cdot\|_E$ satisfying  $x\in D(J)$, $y\in E$, $x^*\leq y^*$ implies $x\in E$ and $\|x\|_E\leq\|y\|_E$.
\end{defi}
Henceforth, we will refer to these as symmetric function spaces,
including the possibility that they are symmetric sequence spaces.

By \cite{GI}, every $\Bc$-bimodule $\Ic\subseteq S(\Bc,\tau)$ can be uniquely described by its characteristic set
\[\mu(\Ic):=\{\mu(A):\, A\in\Ic\}.\]
Alternatively, the correspondence can be seen at the level of function spaces and this
goes further to characterize also symmetric norms.
In particular, given a symmetric operator space $(\Ic,\|\cdot\|_{\Ic})$ in $S(\Bc,\tau)$ consider
the subspace
\[
E_\Ic:=\{x\in D(J):\, x^*\in\mu(\Ic)\}
\]
and for $x\in E_\Ic$ set $\|x\|_{E_\Ic}:=\|A\|_{\Ic}$ for $A\in\Ic$ such that $\mu(A)=x^*$.
Conversely, given a symmetric function
space $(E,\|\cdot\|_E)$, define
\[
\Ic_E:=\{A\in S(\Bc,\tau):\, \mu(A)\in E\},\quad \|A\|_{\Ic_E}:=\|\mu(A)\|_E.
\]
We have the following correspondence between symmetric operator spaces and symmetric function spaces.

\begin{thm}[{\cite[Theorem 3.1.1]{LSZ}; see also \cite[Theorem 8.11]{KS}}]
\label{correspondence}
The map
\[
(\Ic,\|\cdot\|_\Ic)\mapsto(E_\Ic,\|\cdot\|_{E_\Ic})
\]
is a bijection from the set of all symmetric operator
spaces in $S(\Bc,\tau)$ onto the set of all symmetric function spaces in $D(J)$, whose inverse is the map
\[
(E,\|\cdot\|_E)\mapsto(\Ic_E,\|\cdot\|_{\Ic_E}).
\]
\end{thm}

It is proved in \cite[Theorem II.4.1]{KPS} that every symmetric function subspace $E$ of $D(J)$ satisfies
\[L^1\cap L^\infty\subseteq E\subseteq L^1+ L^\infty.\]
Due to Theorem \ref{correspondence}, we have that every symmetric operator space $\Ic$ in $S(\Bc,\tau)$ satisfies
\[L^1(\Bc,\tau)\cap L^\infty(\Bc,\tau)\subseteq\Ic\subseteq L^1(\Bc,\tau)+L^\infty(\Bc,\tau).\]

We will use the H\"older inequality in Banach lattices to obtain the inequality~\eqref{H4} in symmetric operator spaces.

\begin{defi}[{\cite[Definition 1.a.1]{LT}}]
\label{lattice}
A partially ordered Banach space $X$ over the field $\Reals$ is called a Banach lattice, provided
\begin{enumerate}[(i)]
\item\label{latticei}
$x\leq y$ implies $x+z\leq y+z$, for every $x,y,z\in X$,

\item\label{latticeii}
$ax\geq 0$, for every $0\leq x\in X$ and every $a\in (0,\infty)$,

\item\label{latticeiii}
for all $x,y\in X$ there exists a least upper bound $x\vee y\in X$ and a greatest lower bound $x\wedge y\in X$,

\item\label{latticeiv}
$\|x\|_X\leq\|y\|_X$ whenever $|x|\leq|y|$, for $x,y\in X$, where $|x|:=x\vee(-x)$.
\end{enumerate}
\end{defi}
Every symmetric function space $(E,\|\cdot\|_E)$
is a Banach lattice with the partial order defined by pointwise inequality (almost everywhere),
with $\vee$ and $\wedge$ then taken pointwise.
Indeed, the properties \eqref{latticei} and \eqref{latticeii} hold trivially
and the property \eqref{latticeiv} is an immediate consequence of monotonicity
of the decreasing rearrangement and the symmetric property of $E$.
Let $x,y\in E$. Since $0\leq|x\vee y|\leq|x|\vee |y|\leq|x|+|y|$,
by monotonicity of the decreasing rearrangement and the symmetric property of $E$, we get $x\vee y\in E$.
Since $x\wedge y=x+y-x\vee y$, we also have $x\wedge y\in E$. Thus, the property \eqref{latticeiii} also holds.

In a Banach lattice $X$, there is a functional calculus
\[
X^n\ni(x_1,\ldots,x_n)\mapsto f(x_1,\ldots,x_n)\in X
\]
for functions $f:\Reals^n\to\Reals$ that are homogenous of degree $1$
(see~\cite[Theorem 1.d.1]{LT}).
In the case of a symmetric function space, this functional calculus coincides with defining $f(x_1,\ldots,x_n)$ pointwise.

\begin{prop}[{\cite[Proposition 1.d.2(i)]{LT}}]
\label{Holderlattice}
Let $X$ be a Banach lattice. For every $0<\theta<1$ and every $x,y\in X$,
\[\left\||x|^\theta\cdot|y|^{1-\theta}\right\|_X\leq \|x\|_X^\theta\cdot\|y\|_X^{1-\theta},\]
where $|x|^\theta|y|^{1-\theta}\in X$ is given by the functional calculus described above.
\end{prop}

\begin{defi} Let $(E,\|\cdot\|_E)$ be a symmetric function space in $D(J)$.
For $1<p<\infty$,
let
\[
E^{1/p}:=\{x\in D(J):\, |x|^p\in E\},
\]
and for $x\in E^{1/p}$,
let
\begin{equation}
\label{Erootnorm}
\|x\|_{E^{1/p}}:=\||x|^p\|_E^{1/p}.
\end{equation}
\end{defi}

\begin{prop}[{\cite[Proposition 2.23(i)]{ORS}}]
\label{functioncomplete}
Let $E$ be a symmetric function space.
Then, for $1<p<\infty$, $(E^{1/p},\|\cdot\|_{E^{1/p}})$ is a symmetric function space.
\end{prop}

Proposition \ref{functioncomplete}, asserts, in particular,
the completeness of the norm~\eqref{Erootnorm} on $E^{1/p}$. As an immediate consequence of Proposition \ref{Holderlattice}, we obtained the following.
\begin{prop}
\label{HolderlatticeC}
Let $E$ be a symmetric function space. For every $1<p,q<\infty$, with $\frac1p+\frac1q=1$, and every $x\in E^{1/p}, y\in E^{1/q}$, we have $xy\in E$ and
\[\left\|xy\right\|_{E}\leq \||x|^p\|_{E}^{\frac1p}\;\||y|^q\|_{E}^{\frac1q}.\]
\end{prop}

From Proposition \ref{functioncomplete}
and the correspondence described in Theorem~\ref{correspondence}, we immediately obtain the analogue
of Proposition \ref{functioncomplete} for symmetric operator spaces.
\begin{thm}
\label{completeness}
Let $(\Ic,\|\cdot\|_{\Ic})$ be a symmetric operator space.
Then, $\big(\Ic^{1/p},\|\cdot\|_{\Ic^{1/p}}\big)$ defined in~\eqref{rootideal} and~\eqref{rootnorm} is a symmetric operator space.
\end{thm}

Here is the main result of this note.
\begin{thm}
\label{Holderideal}
Let $\Ic$ be a symmetric operator space in $S(\Bc,\tau)$.
Then for every $1<p, q<\infty$, with $\frac1p+\frac1q=1$, and every $A\in \Ic^{1/p}$, $B\in \Ic^{1/q}$,
we have $AB\in\Ic$ and
\begin{equation}
\label{H4again}
\left\|AB\right\|_\Ic\leq 4\,\|A\|_{\Ic^{1/p}}\|B\|_{\Ic^{1/q}}.
\end{equation}
\end{thm}
\begin{proof}
Let $(E,\|\cdot\|_E)=(E_\Ic,\|\cdot\|_{E_\Ic})$ be as in the correspondence from Theorem~\ref{correspondence}.
Suppose firstly that
$\Bc$ is semifinite but not finite (namely, that $\Bc$ is type II$_\infty$ or I$_\infty$).
We can
choose isometries $V_1,V_2\in\Bc$ so that $V_1V_1^*+V_2V_2^*=1$ and then $\tau(V_1DV_1^*)=\tau(D)$ for every $D\in S(\Bc,\tau)$ satisfying $D\ge0$.
If $C,D\in S(\Bc,\tau)$, then by $C\oplus D\in S(\Bc,\tau)$ we will mean the element $V_1CV_1^*+V_2DV_2^*$.

By the properties of the generalized singular values (see \cite[Proposition 2.5]{FK}), we have
\begin{equation}
\label{muAB}
\mu_t(AB)\leq\mu_{t/2}(A)\,\mu_{t/2}(B)=\mu_t(A\oplus A)\,\mu_t(B\oplus B),
\end{equation}
with $A\oplus A\in\Ic^{1/p}$ and $B\oplus B\in\Ic^{1/q}$.
Applying Proposition \ref{HolderlatticeC}
gives
$\mu(A\oplus A)\mu(B\oplus B)\in E$ and
\begin{align*}
\|\mu(A\oplus A)\mu(B\oplus B)\|_E
&\le\|(\mu(A\oplus A))^p\|_E^{1/p}\|(\mu(B\oplus B))^q\|_E^{1/q} \\
&=\|\mu(A\oplus A)\|_{E^{1/p}}\|\mu(B\oplus B)\|_{E^{1/q}}.
\end{align*}
From~\eqref{muAB} and the symmetry condition, we have $\mu(AB)\in E$ and
\[
\|\mu(AB)\|_E
\leq
\|\mu(A\oplus A)\|_{E^{1/p}}\|\mu(B\oplus B)\|_{E^{1/q}}.
\]
Thus, $AB\in\Ic$.
The equality of norms $\|C\|_\Ic=\|\mu(C)\|_E$, for $C\in\Ic$ (and similarly for $C$ in $\Ic^{1/p}$ and $\Ic^{1/q}$)
immediately implies
\begin{equation}
\label{eq:AB}
\|AB\|_\Ic\leq\|A\oplus A\|_{\Ic^{1/p}}\|B\oplus B\|_{\Ic^{1/q}}.
\end{equation}
Since we have
\begin{equation}
\label{eq:AA}
\|A\oplus A\|_{\Ic^{1/p}}\le\|A\oplus0\|_{\Ic^{1/p}}+\|0\oplus A\|_{\Ic^{1/p}}=2\|\mu(A)\|_{E^{1/p}}=2\|A\|_{\Ic^{1/p}}
\end{equation}
and a similar inequality for $B$,
from~\eqref{eq:AB} we get~\eqref{H4again}. This completes the proof when $\Bc$ is not of type II$_1$.

When $\Bc$ has type II$_1$, essentially the same proof works.
However, instead of $A\oplus A$ we use $W_1|A|W_1^*+W_2|A|W_2^*$ where
$W_1$ and $W_2$ partial isometries satisfying $W_i^*W_i=P$,
for $i\in\{1,2\}$,
and $W_1W_1^*+W_2W_2^*=1$ for a projection $P$ of trace $1/2$ in $\Bc$
that commutes with $|A|$ and so that
\[
\mu_r(|A|P)=\begin{cases}
\mu_r(|A|),&0<r< 1/2, \\
0,&1/2\le r\le 1,
\end{cases}
\]
A similar procedure is performed for $B$, but with a projection $Q$ and partial isometries $U_1$ and $U_2$.
Then arguing as above, instead of~\eqref{eq:AB} we get
\[
\|AB\|_\Ic\leq\big\|W_1|A|W_1^*+W_2|A|W_2^*\big\|_{\Ic^{1/p}}\;\big\|U_1|B|U_1^*+U_2|B|U_2^*\,\big\|_{\Ic^{1/q}}.
\]
and instead of~\eqref{eq:AA} we get
\[
\big\|W_1|A|W_1^*+W_2|A|W_2^*\big\|_{\Ic^{1/p}}\le2\|\mu(|A|P)\|_{E^{1/p}}=2\big\||A|P\big\|_{\Ic^{1/p}}
\le2\|A\|_{\Ic^{1/p}},
\]
and a similar inequality for $|B|Q$.
Thus, also in this case we get~\eqref{H4again}.
\end{proof}

\begin{ques}
What is the best constant in~\eqref{H4again}?
In particular, can $4$ be replaced by $1$?
\end{ques}

In the next section, we see that the constant is $1$ in strongly symmetric operator spaces.

\section{The H\"{o}lder inequality in strongly symmetric operator spaces}

For a symmetric Banach function space $E$, recall that we have $E\subseteq L^1(J)+L^\infty(J)$.
We have the notion of the Hardy--Littlewood submajorization:
for $x,y\in E$, we write $x\prec\prec y$ to mean
\[
\int_0^t x^*(s)\,ds\le\int_0^t y^*(s)\,ds, \quad  t>0.
\]
Similarly, for $A,B$ in a symmetric operator space $\Ic \subseteq L^1(\Bc,\tau)+L^\infty(\Bc,\tau)$, we say that $A$ is submajorized by $B$ and write $A\prec\prec B$ if
\[
\int_0^t \mu_s(A)\,ds\le\int_0^t \mu_s(B)\,ds, \quad  t>0.
\]

The following definition is standard.
\begin{defi}
\label{ssdef}
\begin{enumerate}[(i)]
\item
A symmetric function space $E\subseteq L^1(J)+L^\infty(J)$ is said to be
{\em strongly symmetric} if $x,y\in E$, 
$x\prec\prec y$ implies $\|x\|_E\le\|y\|_E$.
\item
A symmetric operator space $(\Ic,\|\cdot\|_\Ic)$ is said to be strongly symmetric if the corresponding symmetric function space $(E_\Ic,\|\cdot\|_{E_\Ic})$ is strongly symmetric. Equivalently, a symmetric operator space $\Ic\subseteq L^1(\Bc,\tau)+L^\infty(\Bc,\tau)$ is strongly symmetric if $A,B\in\Ic$, $A\prec\prec B$ implies $\|A\|_{\Ic}\leq\|B\|_{\Ic}$.
\end{enumerate}
\end{defi}

Now we see that the H\"older inequality holds (with constant $1$) in strongly symmetric operator spaces.
\begin{thm}
Let $(\Ic,\|\cdot\|_\Ic)$ be a strongly symmetric operator space in $S(\Bc,\tau)$.
Then for every $1<p,q<\infty$, with $\frac1p+\frac1q=1$,
and every $A\in\Ic^{1/p}$, $B\in\Ic^{1/q}$, we have $AB\in\Ic$ and
\begin{equation}\label{eq:ABI}
\|AB\|_\Ic\le\|A\|_{\Ic^{1/p}}\,\|B\|_{\Ic^{1/q}}.
\end{equation}
\end{thm}
\begin{proof}
Let $(E,\|\cdot\|_E)=(E_\Ic,\|\cdot\|_{E_\Ic})$.
Making use of Proposition~\ref{HolderlatticeC}
and the functional calculus described above it, we have $\mu(A)\mu(B)\in E$ and
\begin{equation}\label{eq:ABstr}
\|\mu(A)\mu(B)\|_E\le\|\mu(A)\|_{E^{1/p}}\|\mu(B)\|_{E^{1/q}}.
\end{equation}
By Theorem \ref{Holderideal},
we also have $\mu(AB)\in E$ and, by~\cite[Theorem 4.2]{FK} (with $f(x)=x$) we have $\mu(AB)\prec\prec\mu(A)\mu(B)$.
Since $E$ is strongly symmetric, we have
\[
\|\mu(AB)\|_E\le\|\mu(A)\mu(B)\|_E.
\]
Now~\eqref{eq:ABI} follows from this and~\eqref{eq:ABstr}.
\end{proof}

\section*{Acknowledgment}

The authors are grateful to Marius Junge helpful suggestions.

\bibliographystyle{plain}

\begin{thebibliography}{9}
\bibitem{BS}
C.~Bennett, R.~Sharpley, {\it Interpolation of operators},
Pure and Applied Mathematics, 129. Academic Press, Inc., Boston, MA, 1988.

\bibitem{DS}
K. Dykema, A. Skripka, {\it Perturbation formulas for traces on normed ideals,} Comm. Math. Phys., {\bf 325} (2014), no. 3, 1107--1138.


\bibitem{FK}
T.~Fack, H.~Kosaki, {\it Generalized s-numbers of $\tau$-measurable operators,} Pacific J. Math. {\bf 123} (1986), no. 2, 269--300.

\bibitem{GI}
D.~Guido, T.~Isola, {\it Singular traces on semifinite von Neumann algebras},
J. Funct. Anal. {\bf 134} (1995), no. 2, 451--485.

\bibitem{KS}
N.~J.~Kalton, F.~A.~Sukochev, {\it Symmetric norms and spaces of operators,} J. Reine Angew. Math. {\bf 621} (2008), 81--121.

\bibitem{KPS}
S.~G.~Krein, Yu.~I.~Petunin, E.~M.~Semenov, {\it Interpolation of linear operators.}
Translations of Mathematical Monographs, 54. American Mathematical Society, Providence, R.I., 1982.

\bibitem{LT}
J.~Lindenstrauss, L.~Tzafriri, {\it Classical Banach spaces. II. Function spaces,} Ergebnisse der Mathematik und ihrer Grenzgebiete
97, Springer-Verlag, Berlin-New York, 1979.

\bibitem{LSZ}
S.~Lord, F.~Sukochev, D.~Zanin, {\it Singular Traces}, de Gruyter Studies in Mathematics, 46,
Walter de Gruyter \& Co., Berlin, 2012.


\bibitem{ORS}
S.~Okada, W.~J.~Ricker, E.~A.~S\`{a}nchez P\`{e}rez, Optimal domain and integral extension of operators.
Acting in function spaces. Operator Theory: Advances and Applications, 180. Birkhäuser Verlag, Basel, 2008.

\bibitem{Russu}
G.~I.~Russu, {\it Symmetric spaces of functions that do not have the majorization property,} Mat. Issled. {\bf 4} (1969), vyp. 4 (14), 82--93. (Russian)

\bibitem{SSS}
A.~A.~Sedaev, E.~M.~Semenov, F.~A.~Sukochev, {\it Fully symmetric function spaces without an equivalent Fatou norm,} Positivity, DOI 10.1007/s11117-014-0305-5.

\bibitem{S}
A.~Skripka, {\it Trace formulas for multivariate operator functions,} Integral Equations Operator Theory, 
81 (2015), no.~4, 559--580.

\end{thebibliography}

\end{document}